\documentclass[a4paper,12pt]{amsart}
\usepackage{amsmath}
\usepackage{amssymb}
\usepackage{mathrsfs}
\usepackage{cite}
\usepackage{enumerate}
\usepackage{ifthen}
\usepackage{graphicx}
\usepackage{hyperref}
\hypersetup{
    colorlinks=true,
    citecolor=blue,
    linkcolor=blue,
    urlcolor=blue
}
\usepackage[T1]{fontenc} %skandit
\usepackage{tabularx}
\usepackage{multirow}
\usepackage{color,soul}
\usepackage{tikz}
\usepackage{pgfplots} % Optional, for more advanced plotting
\usepackage{amsthm}
\newtheorem{pro}{Proposition}

\nonstopmode \numberwithin{equation}{section}
\newtheorem{thm}{Theorem}[section]

\newtheorem{lem}{Lemma}[section]

\theoremstyle{definition}

\newtheorem{rem}{Remark}[section]

\newtheorem{innercustomthm}{Theorem}
\newenvironment{customthm}[1]
  {\renewcommand\theinnercustomthm{#1}\innercustomthm}
  {\endinnercustomthm}
\newcounter{minutes}
\divide\time by 60
\newcounter{hours}
\multiply\time by 60
\addtocounter{minutes}{-\time}
\newcounter {own}
\def\theown {\thesection       .\arabic{own}}

\newenvironment{pf}[1][]{%
 \vskip 3mm
 \noindent
 \ifthenelse{\equal{#1}{}}%
  {{\slshape Proof. }}%
  {{\slshape #1.} }%
 }%
{\qed\bigskip}

\begin{document}
\title{ On the logarithmic coefficients of Ma-Minda type convex functions}

\author{Md Firoz Ali}
\address{Md Firoz Ali,
 National Institute Of Technology Durgapur, West Bengal, India}
\email{ali.firoz89@gmail.com, fali.maths@nitdgp.ac.in}

\author{Lokenath Thakur}
\address{Lokenath Thakur, National Institute Of Technology Durgapur, West Bengal, India}
\email{lokenaththakur1729@gmail.com}



\def\thefootnote{}
\footnotetext{ {\tiny File:~\jobname.tex,
printed: \number\year-\number\month-\number\day,
          \thehours.\ifnum\theminutes<10{0}\fi\theminutes }
} \makeatletter\def\thefootnote{\@arabic\c@footnote}\makeatother

\begin{abstract}
In this paper, we investigate three specific subclasses of Ma-Minda type convex functions: namely, convex functions of order $\alpha$, Janowski convex functions, and Robertson functions of normalized analytic functions defined in the open unit disk.
For these classes, we establish logarithmic coefficient inequalities concerning both individual coefficient estimates and weighted series. The results presented here correct some earlier erroneous results and extend several previously known ones.
\end{abstract}
\thanks{}

\maketitle
\pagestyle{myheadings}
\markboth{Md Firoz Ali and Lokenath Thakur }{ A study of logarithmic coefficients for certain Ma-Minda type functions}

\section{Introduction}
Let $\mathcal{H}$ denote the class of analytic functions defined in the open unit disk $\mathbb{D} = \{z \in \mathbb{C} : |z| < 1\}$. We define $\mathcal{A}$ as the subclass of $\mathcal{H}$ consisting of functions normalized by the power series expansion
\begin{equation}\label{T-05}
f(z) = z + \sum_{n=2}^\infty a_n z^n,
\end{equation}
and let $\mathcal{S} \subset \mathcal{A}$ be the class of functions that are univalent (one-to-one) in $\mathbb{D}$.
The study of coefficient bounds for $\mathcal{S}$ is a cornerstone of geometric function theory. In $1916$, Bieberbach \cite{1916-Bieberbach} demonstrated that for any $f \in \mathcal{S}$, the second coefficient satisfies $|a_2| \le 2$, with equality occurring if and only if $f$ is a rotation of the Koebe function $k(z) = z/(1-z)^2$. This observation led to the famous Bieberbach conjecture $|a_n| \le n$ for all $n \ge 3$, with the Koebe function again serving as the extremal case.\\

Central to the eventual resolution of the Bieberbach conjecture are the logarithmic coefficients $\gamma_n$ of $f \in \mathcal{S}$, defined by
\begin{equation}\label{T-71}
F_f(z)=\log\frac{f(z)}{z}=2\sum\limits_{n=1}^{\infty}\gamma_nz^n.
\end{equation}
The importance of these coefficients was first recognized by Bazilevi\'{c} \cite{1965-Bazilevich}. In subsequent work, Bazilevi\'{c} \cite{1967-Bazilevich} studied the Dirichlet-type series $\sum_{n=1}^{\infty} n|\gamma_n|^2 r^{2n}$, which relates to the area of the image of the disk $|z| < r<1$ under the mapping 
$\frac{1}{2}F_f(z)$ for $f\in\mathcal{S}$. A pivotal result in this domain is the de Branges inequality (resolving the earlier Milin conjecture) which asserts that if $f \in \mathcal{S}$ and the coefficients $\gamma_n$ are defined as in \eqref{T-71}, then 
$$
\sum_{k=1}^n (n-k+1)|\gamma_n|^2\le\sum_{k=1}^n \frac{n-k+1}{k}, \quad n=1,2,\cdots,
$$
with equality holding precisely for Koebe function or one of its rotations (see \cite{1985-Branges}). This inequality was instrumental in de Branges's proof of the long-standing Bieberbach conjecture. Furthermore, it inspired various subsequent inequalities, such as the well-known bound (see \cite{1979-Duren})
$$
\sum_{k=1}^\infty|\gamma_n|^2\le\sum_{k=1}^\infty \frac{1}{k^2}=\frac{\pi^2}{6}.
$$
Along similar lines, Roth \cite{Roth-2007} established a sharp inequality for functions $f \in \mathcal{S}$:
\begin{equation}\label{T-08}
\sum_{n=1}^{\infty} \lambda_n |\gamma_n|^2 \le \sum_{n=1}^{\infty} \frac{\lambda_n}{n^2}, \quad \lambda_n= \left(\frac{n}{n+1}\right)^2.
\end{equation}
This inequality has become a source of many new inequalities for the logarithmic coefficients of univalent functions (see \cite{Ponnusamy-Sharma- Wirths-2020}). Roth's inequality \eqref{T-08} was subsequently extended by Ponnusamy and Sugawa \cite{ponnusami-sugawa-2021}, who established a family of sharp inequalities for logarithmic coefficients using convex sequences as weights.\\

While the average behavior of logarithmic coefficients $\gamma_n$ has been studied extensively, (see \cite{Ali-Vasudevarao-2018, 1979-Duren}), precise upper bounds on $|\gamma_n|$ for functions in the class $\mathcal{S}$ remain scarce. The Koebe function $k(z) = z/(1-z)^2$, whose logarithmic coefficients are given by $\gamma_n = 1/n$, typically serves as the extremal function for many problems in $\mathcal{S}$. This might lead to the conjecture that $|\gamma_n| \leq 1/n$ holds generally for functions in $\mathcal{S}$. However, this conjecture fails dramatically, not only in precise values but also in terms of asymptotic behavior.
%In fact, \cite[Theorem 8.4]{1983-Duren} demonstrates the existence of a bounded function $f \in \mathcal{S}$ whose logarithmic coefficients satisfy $\gamma_n \neq O(n^{-0.83})$, showing that the Koebe function's coefficients do not universally dominate.
However, one of the reasons the logarithmic coefficients have received more attention is because, although the sharp bound  
\begin{align*}
|\gamma_1|\le 1\quad \mathrm{and}\quad |\gamma_2|\le \frac{1}{2}\left(1+2 e^{-2}\right)
\end{align*}
for the class $\mathcal{S}$ is known, the problem of determining the sharp bounds for $|\gamma_n|$ when $n \ge 3$ remains an open challenge.\\

For $f, g \in \mathcal{H}$, we say $f$ is subordinate to $g$, denoted $f \prec g$, if there exists a function $\omega: \mathbb{D} \to \mathbb{D}$ with $\omega(0) = 0$ such that $f(z) = g(\omega(z))$ for all $z \in \mathbb{D}$. If $g$ is univalent, then $f \prec g$ if and only if $f(0) = g(0)$ and $f(\mathbb{D}) \subset g(\mathbb{D})$.
For \( \varphi \in \mathcal{H} \) with $\varphi(0)=1$, let the classes \( \mathcal{S}^*(\varphi) \) and $\mathcal{C}(\varphi)$ be defined by
$$
\mathcal{S}^*(\varphi) = \left\{ f \in \mathcal{A} : \frac{zf'(z)}{f(z)} \prec \varphi(z) \right\},~
\mathrm{and}~~
\mathcal{C}(\varphi) = \left\{ f \in \mathcal{A} : 1 + \frac{zf''(z)}{f'(z)} \prec \varphi(z) \right\},
$$
respectively. The classes $\mathcal{S}^*(\varphi)$ and $\mathcal{C}(\varphi)$ were introduced by Ma and Minda \cite{ma-minda-1992} under the additional hypotheses that $ \varphi $ is univalent with a positive real part in $ \mathbb{D} $, and that $ \varphi(\mathbb{D}) $ is symmetric with respect to the real axis and starlike with respect to $ \varphi(0) = 1 $, with $ \varphi'(0) > 0 $. The classes $ \mathcal{S}^*(\varphi) $ and $\mathcal{C}(\varphi)$ are called the Ma-Minda classes of starlike and convex functions, respectively. It is important to note that $ f \in \mathcal{C}(\varphi) $ if and only if $zf'(z) \in \mathcal{S}^*(\varphi)$.\\

For different choices of $\varphi$, the classes $ \mathcal{S}^*(\varphi) $ and $\mathcal{C}(\varphi)$ reduce to several well-known classes which were earlier introduced and studied for their geometric and analytic properties. For instance, if $\varphi(z)=(1+z)/(1-z)$, then $\mathcal{S}^*(\varphi) =: \mathcal{S}^*$ and $\mathcal{C}(\varphi) =: \mathcal{C}$, representing the usual classes of starlike and convex functions, respectively.
If $\varphi(z)=(1+(1-2\alpha)z)/(1-z)$ with $0 \le \alpha < 1$, then $\mathcal{S}^*(\varphi) =: \mathcal{S}^*(\alpha)$ and $\mathcal{C}(\varphi) =: \mathcal{C}(\alpha)$, which are the well-known classes of starlike and convex functions of order $\alpha$, introduced in \cite{Robertson-1935}. Similarly, if $\varphi(z)=(1+Az)/(1+Bz)$ with $-1 \le B < A \le 1$, then $\mathcal{S}^*(\varphi) =: \mathcal{S}^*(A,B)$ and $\mathcal{C}(\varphi) =: \mathcal{C}(A,B)$, denoting the classes of Janowski starlike and Janowski convex functions introduced in \cite{Janowski1,Janowski2}.
If $\varphi=(1+(c-1)z)/(1-z)$ with $0<c\le 3$, then $\mathcal{C}(\varphi)$ reduces to the class $\mathcal{F}(c)$ (see \cite{Ponnusamy-Sharma- Wirths-2020}). For $c\in(0,2]$, if we take $\alpha = 1 - c/2$, the family \(\mathcal{F}(c)\) coincides with the class $\mathcal{C}(\alpha)$ of convex functions of order \(\alpha\). For \(c = 3\), the class \(\mathcal{F}(3)\) has received considerable attention in recent years (see \cite{Ponnusamy-Sahoo-Yanagihara-2014}). Functions in \(\mathcal{F}(3)\), are known to be convex in one direction, and are hence univalent and close-to-convex, but they are not necessarily starlike in \(\mathbb{D}\) (see \cite{Umezawa-1952}).
Furthermore, if $\varphi=(1+e^{-2i\alpha }z)/(1-z)$ with $|\alpha|<\pi/2$, then $\mathcal{C}(\varphi)$ reduces to the Robertson class $\mathcal{S}_\alpha$, introduced by Robertson \cite{Robertson-1969}. Unlike convex or starlike functions, members of \(\mathcal{S}_{\alpha}\) are not necessarily univalent. Robertson \cite{Robertson-1969} himself noted that univalence in \(\mathcal{S}_{\alpha}\) is not guaranteed for all \(\alpha\). Indeed, it is now well known that functions in \(\mathcal{S}_{\alpha}\) are univalent if $0 < \cos \alpha \leq \frac{1}{2}$.\\

%\begin{itemize}
%\item $\mathcal{S}^*\left(\frac{1+z}{1-z}\right) = \mathcal{S}^*$ (starlike functions), $\mathcal{C}\left(\frac{1+z}{1-z}\right) = \mathcal{C}$ (convex functions).
%\vspace{2 mm}
%\item $\mathcal{S}^*\left(\frac{1+Az}{1+Bz}\right) = \mathcal{S}^*(A,B)$,     $\mathcal{C}\left(\frac{1+Az}{1+Bz}\right) = \mathcal{C}(A,B)$, for $-1 \leq B < A \leq 1$ (Janowski classes).
%\vspace{2 mm}
%\item $\mathcal{C}\left(1 + \frac{cz}{1-z}\right) = \mathcal{F}(c)$, for $c \in (0,3]$.
%\vspace{2 mm}
%\item $\mathcal{C}(z + \sqrt{1+z^2}) = \left\{ f \in \mathcal{A} : 1 + \frac{zf''(z)}{f'(z)} \prec z + \sqrt{1+z^2} \right\}$.
%\item $\mathcal{S}_\alpha = \left\{ f \in \mathcal{A} : 1 + \frac{zf''(z)}{f'(z)} \prec \frac{1+e^{-2i\alpha}z}{1-z} \right\}$, for $-\pi/2 < \alpha < \pi/2$ (Robertson class).\\
%\end{itemize}

In this article, we investigate inequalities concerning the logarithmic coefficients of functions belonging to the classes $\mathcal{F}(c)$, $\mathcal{C}(A,B)$, and $\mathcal{S}_\alpha$. Our primary objective is to establish sharp upper bounds for these coefficients, thereby extending several results currently existing in the literature. Furthermore, we identify certain inaccuracies in previous studies and provide corrected versions of those results. The remainder of this paper is structured as follows. In Section \ref{T-section-001}, we provide the formal statements of our main theorems, detailing the specific coefficient bounds obtained for each class. In Section \ref{T-section-002}, we present the necessary auxiliary lemmas that serve as the foundation for our subsequent analysis. Finally, the proofs of the main theorems are developed in Section \ref{T-section-003}.

\section{Main Results}\label{T-section-001}

\subsection{The class   $\mathcal{F}(c)$}
Recall that a function $f\in\mathcal{A}$ belongs to $\mathcal{F}(c)$, $c\in(0,3]$ if  $f\in\mathcal{C}(\varphi)$ with $\varphi=(1+(c-1)z)/(1-z)$. In other words,  a function $f\in\mathcal{A}$ is in $\mathcal{F}(c)$ if and only if
\[
\operatorname{Re}\left(1 + \frac{zf''(z)}{f'(z)}\right) > 1 - \frac{c}{2}.
\]
%For $c\in(0,2]$, if we take $\alpha = 1 - c/2$, the family \(\mathcal{F}(c)\) coincides with the well-known class $\mathcal{C}(\alpha)$ of convex functions of order \(\alpha\). For \(c = 3\), the class \(\mathcal{F}(3)\) has received considerable attention in recent years (see \cite{Ponnusamy-Sahoo-Yanagihara-2014}). Functions in \(\mathcal{F}(3)\), are known to be convex in one direction, hence univalent and close-to-convex, but not necessarily starlike in \(\mathbb{D}\) (see \cite{Umezawa-1952}).
For a wide range of extremal problems within the class $\mathcal{F}(c)$, the function 
\begin{align}\label{T-fun-001}
    f_c(z) = 
\begin{cases}
\displaystyle \frac{(1 - z)^{1-c} - 1}{c - 1}, &~~~c\neq1,\\
\displaystyle -\log(1-z), &~~~c=1.
\end{cases}
\end{align}
often serves as the extremal function. With the help of the function $f_c$, we define a new function $\psi(z)$ as given below:
\begin{align}\label{T-fun-002}
    \psi(z) = \frac{zf'_c(z)}{f_c(z)}=
\begin{cases}
   \displaystyle \frac{(c - 1)z}{1 - z} \cdot \frac{1}{1 - (1 - z)^{c-1}}, & c\neq1,\\[2mm]
   \displaystyle \frac{-z}{(1-z)\log(1-z)}, & c=1.
\end{cases}
\end{align}
This function will play an important role in the present article.\\

Recently, Cho et al. \cite{cho-Alimohammadi-2021} studied the logarithmic coefficients for functions in the class $\mathcal{F}(c)$ and proved the following theorem.

\begin{customthm}{A}\cite[Theorem 2.1 ]{cho-Alimohammadi-2021}\label{T-thm-A}
For $f\in \mathcal{F}(c)$ with $c \in (0, 0.656] \cup \{2\}$, the logarithmic coefficients of $f$ satisfies the inequality
\begin{equation}\label{eq:main-inequality}
|\gamma_n| \leq \frac{|D_n|}{2n}, \quad n = 1, 2, 3, \dots,
\end{equation}
and
$$
\sum_{n=1}^{\infty} |\gamma_n|^2 \leq \frac{1}{4} \sum_{n=1}^{\infty} \frac{|D_n|^2}{ n^2},
$$
where $D_n$ are the Taylor coefficients of $\psi(z)$ given by \eqref{T-fun-002}. Further, both the inequalities are sharp.
\end{customthm}

The proof of Theorem \ref{T-thm-A} (specially the estimate \eqref{eq:main-inequality}) completely relies on the convexity of $\psi(z)$ defined by \eqref{T-fun-002}. The authors of \cite{cho-Alimohammadi-2021} claimed to verify the convexity of $\psi(z)$ pictorially (see \cite[Fig. 1]{cho-Alimohammadi-2021}) using \textsc{Maple}\textsuperscript{\texttrademark} software by showing that
\begin{equation}\label{eq:convexity-condition}
\operatorname{Re}\Psi(z)= \operatorname{Re} \left( 1 + \frac{z \psi''(z)}{\psi'(z)} \right) > 0 \quad \text{for} ~ z \in \mathbb{D}.
\end{equation}
However, this claim is incorrect. Indeed, first we prove that the function $\psi(z)$ is not convex for $c \in (1/2, 2)$.

\begin{pro}\label{T-prop-001}
The function $\psi(z)$ defined by \eqref{T-fun-002} is not convex when $c \in (1/2, 2)$.
\end{pro}

On the other hand, by numerical computation, we demonstrate that $\psi(z)$ defined by \eqref{T-fun-002} fails to be convex for several values of $c$ in $(0, 1/2]$. To show that $\psi(z)$ is not convex for a certain value of $c\in (0, 1/2]$, we show that
$$
\operatorname{Re}\Psi(e^{i\theta})= \operatorname{Re}\left( 1 + \frac{e^{i\theta}\psi''(e^{i\theta})}{\psi'(e^{i\theta})} \right)<0
$$
for some values of $z=e^{i\theta}$, $0\le \theta<2\pi$ on the unit circle $|z|=1$.
In \textbf{Table 1}, we demonstrate that $\psi(z)$ is not convex for selected values of $c \in (0, 1/2]$. Consequently, for the values of $c$ mentioned in Table 1, the function $\psi(z)$ fails to be convex. This suggests that $\psi(z)$ may lack convexity for the entire interval $c \in (0, 1/2]$, thereby implying that the estimate \eqref{eq:main-inequality} is not generally valid. Here it is pertinent to mention that Ponnusamy et al. \cite{ponnusamy-Sharma-Wirths-2018} established sharp estimates for $|\gamma_n|$ when $n=1, 2, 3$ for all $c \in (0,3]$, while obtaining sharp bounds for $|\gamma_4|$ and $|\gamma_5|$ only under the restricted conditions $c \le 144/55$ and $c \le 80/61$, respectively.\\

\begin{table}[h]
\centering
\caption{} 
\begin{tabular}{|c|c|c|}
\hline
$c$ & $\theta ~~ (0\le \theta<2\pi)$ & $\operatorname{Re}\Psi(e^{i\theta})$ \\[2mm] \hline
0.1 & $(2-10^{-20} )\pi$ & $-1.66\times 10^{6} $ \\[2mm] \hline
0.15 & $(2-10^{-31} )\pi$ & $-1.02\times 10^{24} $ \\[2mm] \hline
0.2 & $(2-10^{-28} )\pi$ &$-1.077\times 10^{22} $  \\[2mm] \hline
0.25 & $(2-10^{-13} )\pi$ & $-375.774 $ \\[2mm] \hline
0.3 & $(2-10^{-30} )\pi$ &$-9.57\times 10^{31} $  \\[2mm] \hline
0.35 & $(2-10^{-17} )\pi$ & $-4.65\times 10^{12} $ \\[2mm] \hline
0.4 & $(2-10^{-20} )\pi$ & $-2.05\times 10^{19} $ \\[2mm] \hline
0.45 & $(2-10^{-19} )\pi$ & $-3.41\times 10^{19} $ \\[2mm] \hline
0.4 & $(2-10^{-20} )\pi$ & $-2.05\times 10^{19} $ \\[2mm] \hline
0.5 & $(2-10^{-20} )\pi$ & $-2.18\times 10^{13} $ \\[2mm] \hline
\end{tabular}
\end{table}

Later on, Allu \emph{et al.} \cite{Allu-Sharma-2024} also studied the logarithmic coefficients of functions in $ \mathcal{F}(c)$ extending the results of Cho et al. \cite{cho-Alimohammadi-2021} and proved the following logarithmic coefficient inequalities.

\begin{customthm}{B}\cite{Allu-Sharma-2024}\label{T-thm-B}
Let \( f \in \mathcal{F}(c) \) for \( c \in (0, 0.656] \cup \{2\} \). Then the logarithmic coefficients \(\gamma_n\) of \( f \) satisfy the inequalities
\[
\sum_{n=1}^{\infty} n^2 |\gamma_n|^2 \leq \frac{1}{4} \sum_{n=1}^{\infty} |D_n|^2,
\]
and
\[
\sum_{n=1}^{\infty} (n+1)^t |\gamma_n|^2 \leq \frac{1}{4} \sum_{n=1}^{\infty} \frac{(n+1)^t}{n^2} |D_n|^2 \quad \text{for } t \leq 2,
\]
where $D_n$ are the Taylor's coefficients of $\psi(z)$ given by \eqref{T-fun-002}. The first inequality is sharp.
\end{customthm}

In the proof of Theorem \ref{T-thm-B}, the authors of \cite{Allu-Sharma-2024} used the convexity of the function $\psi(z)$ defined by \eqref{T-fun-002} which compelled them to restrict the result for \( c \in (0, 0.656] \cup \{2\} \). But after a close inspection, we observed that the convexity of $\psi(z)$ is not necessary.
In the next theorem, we extend Theorem \ref{T-thm-B} with $c \in (0,2]$.
%improving upon the previously known result. This broader range enhances the applicability of the theorem to a wider class of functions in $ \mathcal{F}(c) $, thereby strengthening the inequalities for the logarithmic coefficients $ \gamma_n $.
 
\begin{thm}\label{T-thm-001}
Let $ f \in \mathcal{F}(c) $ for $ c \in (0, 2]$. Then the logarithmic coefficients $\gamma_n$ of $ f $ satisfy the inequalities  
$$
\sum_{n=1}^{\infty} n^2 |\gamma_n|^2 \leq \frac{1}{4} \sum_{n=1}^{\infty} |D_n|^2,
$$
and  
$$
\sum_{n=1}^{\infty} (n+1)^t |\gamma_n|^2 \leq \frac{1}{4} \sum_{n=1}^{\infty} \frac{(n+1)^t}{n^2} |D_n|^2 \quad \text{for } t \leq 2,
$$
where the $D_n$ are the Taylor coefficients of $\psi(z)$ given by \eqref{T-fun-002}. Moreover, the first inequality is sharp.
\end{thm}

We note that if we take $\alpha = 1 - c/2$, then the class $\mathcal{F}(c)$ for $c \in (0, 2]$ reduces to the standard class $\mathcal{C}(\alpha)$ of convex functions of order $\alpha$, where $\alpha \in [0, 1)$.

%%Setting $\alpha=1-\frac{c}{2}$ and $t=0$ in the second inequality of Theorem \ref{T-thm-001}, we obtain the following result.
%%\begin{cor}
%%Let \( f \in C(\alpha) \) with $\alpha\in[0,1)$. Then, the logarithmic coefficients $\gamma_n$ of \( f \) satisfy the inequalities
%%\[
%%\sum_{n=1}^{\infty} |\gamma_n|^2 \leq \frac{1}{4} \sum_{n=1}^{\infty} \frac{|G_n|^2}{n^2},
%%\]
%%where \( G_n \) are the Taylor's coefficients of
%%\[
%%\psi_\alpha(z) = \frac{(1 - 2\alpha)z}{1 - z} \cdot \frac{1}{1 - (1 - z)^{1-2\alpha}} =: 1 + \sum_{n=1}^{\infty} G_n z^n.
%%\]
%%The inequality is sharp for the function
%%\[f(z) = \frac{(1 - z)^{2\alpha - 1} - 1}{1 - 2\alpha}.\]
%%\end{cor}

\subsection{The class $\mathcal{C}(A,B)$}
 For \(-1 \leq B < A \leq 1\), the class \(\mathcal{C}(A,B)\) of \emph{Janowski convex functions} consists of functions \(f \in \mathcal{A}\) satisfying the subordination relation
\[
1 + \frac{zf''(z)}{f'(z)} \prec \frac{1 + Az}{1 + Bz}.
\]
%The class \(\mathcal{C}(A,B)\) was first introduced and systematically studied by Janowski \cite{Janowski1,Janowski2}. A necessary and sufficient condition for membership in \(\mathcal{C}(A,B)\) was later given by Silverman and Silvia \cite{SilvermanSilvia1985}.
It is an easy exercise to see that the function $K_{A,B}$ defined by
\begin{align}\label{T-fun-10}
K_{A,B}(z) =
\begin{cases}
\displaystyle \frac{1}{A} \left( (1+Bz)^{A/B} - 1 \right), & A \neq 0, \; B \neq 0, \\[10pt]
\displaystyle \frac{1}{B} \log(1+Bz), & A = 0, \\[10pt]
\displaystyle \frac{1}{A} \left( e^{Az} - 1 \right), & B = 0,
\end{cases}
\end{align}
is a function in $\mathcal{C}(A,B)$. Using the function, we define another function $\psi_{A,B}$ as 
\begin{align}\label{T-70}
    \psi_{A,B}=\frac{zK'_{A,B}(z)}{K_{A,B}(z)}=: 1+\sum_{n=1}^{\infty}E_nz^n.
\end{align}

% Recently, Alimohammadi et al. \cite{cho-Alimohammadi-2021} studied logarithmic coefficients for the class $\mathcal{C}(0,B)$ and proved the following result.

% \begin{customthm}{C}\cite{cho-Alimohammadi-2021}\label{T-thm-C}
%  Let the function \( f \in \mathcal{C}(0, B) \) for \(-0.99 \leq B < 0\). Then, the logarithmic coefficients of \( f \) satisfy the inequalities

% \[
% |\gamma_n| \leq \frac{|L_n|}{2n}, \quad n \in \mathbb{N},
% \]

% and

% \[
% \sum_{n=1}^{\infty} |\gamma_n|^2 \leq \frac{1}{4} \sum_{n=1}^{\infty} \frac{|L_n|^2}{n^2},
% \]

% where \( L_n = \psi_B^{(n)}(0)/n! \) and

% \[
% \psi_{B}(z) = \frac{Bz}{(1 + Bz) \log(1 + Bz)} = 1 + \sum_{n=1}^{\infty} L_n z^n.
% \]
% These inequalities are sharp. 
% \end{customthm}
% The proof of Theorem \ref{T-thm-C} completely relies on the convexity of $\psi_B(z)$. The authors of \cite{cho-Alimohammadi-2021} claimed to verify this property pictorially (see \cite[Fig. 2]{cho-Alimohammadi-2021}) using \textsc{Maple}\textsuperscript{\texttrademark} software by showing that
% \begin{equation}\label{eq:convexity-condition}
% \operatorname{Re}\Psi_B(z)= \operatorname{Re} \left( 1 + \frac{z \psi_B''(z)}{\psi_B'(z)} \right) > 0 \quad \text{for} \quad z \in \mathbb{D}.
% \end{equation}
% However, this claim is incorrect.  For $c=1$, we get $\psi(Bz)=\psi_B(z)$ with $B\in[-1,0)$.
% In proposition \ref{T-prop-001}, we have proved that the function $\psi(z)$ \eqref{T-fun-002} with $c=1$ is not convex in $\mathbb{D}$ and therefore $\psi_B(z)$ is not a convex function.\\

It is worth noting that Cho et al. \cite[Theorem 2.2]{cho-Alimohammadi-2021} studied the logarithmic coefficients for the class $\mathcal{C}(0,B)$ where $B \in [-0.99, 0)$, establishing results analogous to Theorem \ref{T-thm-A}.

\begin{customthm}{C}\cite[Theorem 2.2]{cho-Alimohammadi-2021}\label{T-thm-C}
Let \( f \in \mathcal{C}[0, B] \) for \(-0.99 \leq B < 0\). Then, the logarithmic coefficients $\gamma_n$ of \( f \) satisfy the inequalities
$$
|\gamma_n| \leq \frac{|E_n|}{2n}, \quad n \in \mathbb{N},
$$
and
\[
\sum_{n=1}^\infty |\gamma_n|^2 \leq \frac{1}{4} \sum_{n=1}^\infty \frac{|E_n|^2}{n^2},
\]
where $ E_n $ are the Taylor coefficients of $\psi_{0,B}$ defined by \eqref{T-70} and the inequalities are sharp.
\end{customthm}

The proof of Theorem \ref{T-thm-C} relies on the convexity of $\psi_{0,B}(z)$ defined in \eqref{T-70}. The authors claimed to verify this property pictorially (see \cite[Fig. 2]{cho-Alimohammadi-2021}) using \textsc{Maple}\textsuperscript{\texttrademark}. From the first inequality of Theorem \ref{T-thm-C}, it follows that
$$|\gamma_1| \le \frac{|B|}{4}, \quad |\gamma_2| \le \frac{5|B|^2}{48}, \quad |\gamma_3| \le \frac{|B|^3}{16}
$$
for functions in the class $\mathcal{C}[0, B]$ with $-0.99 \leq B < 0$. However, these estimates are not universally correct, as we demonstrate in Theorem \ref{T-thm-20}. This suggests that the underlying claim regarding the convexity of $\psi_{0,B}(z)$ is incorrect.\\

In the next theorem, we obtain sharp estimates of initial logarithmic coefficients for functions in the class $\mathcal{C}(A,B)$.

\begin{thm}\label{T-thm-20}
Let \( f \in \mathcal{C}(A,B) \) with \( -1 \leq B < A \leq 1 \). Then the logarithmic coefficients \( \gamma_n \) of \( f(z) \) satisfy
\begin{align*}
|\gamma_1| &\leq \frac{A - B}{4},\\
|\gamma_2| &\leq 
\begin{cases}
\displaystyle \frac{A-B}{12} & ~~\text{for}~ |A-5B|\le 4,\\[2mm]
\displaystyle \frac{A-B}{48}|A-5B| & ~~\text{for}~ |A-5B|>4,
 \end{cases}\\[3mm]
|\gamma_3| &\leq 
\begin{cases}
\displaystyle \frac{A - B}{24}, &~~\text{for}~(\mu,\nu)\in D_1\cup D_2, \\[2mm]
\displaystyle \frac{A - B}{48}|B(A - 3B)|, & ~~\text{for}~(\mu,\nu)\in D_6, \\[2mm]
\displaystyle \frac{A-B}{72\sqrt{3}} \cdot \frac{\left( |A - 5B| + 2 \right)^{3/2}}{\sqrt{|A - 5B| + 2 + 3B^{2} - AB}}  & ~~\text{for}~(\mu,\nu)\in D_8\cup D_9,
\end{cases}
\end{align*}
where \(\mu = \frac{1}{2}(A - 5B)\) and \(\nu = \frac{1}{2}(3B^2 - AB)\) and the sets $D_1,D_2, D_6, D_8,D_9$ are defined as in Lemma \ref{T-34}.
All the estimates are sharp.
\end{thm}

On the other hand, Allu and Sharma \cite{Allu-Sharma-2024} proved logarithmic coefficient inequalities similar to Theorem \ref{T-thm-B} for functions in the classes $\mathcal{C}(0,B)$ with $B\in [-0.99,0)$ and $\mathcal{C}(A,0)$ with $A\in (0,1]$.
In the next theorem, we prove logarithmic coefficient inequalities similar to Theorem \ref{T-thm-B} for functions in the class $\mathcal{C}(A,B)$ with $-1\leq B< A \leq 1$, which generalize the previous results of Allu and Sharma \cite{Allu-Sharma-2024}.

\begin{thm}\label{T-thm-10}
Let $ f \in \mathcal{C}(A,B) $ with $-1\leq B< A \leq 1$. Then the logarithmic coefficients $\gamma_n$ of $ f $ satisfy the inequalities  
$$
\sum_{n=1}^{\infty} n^2 |\gamma_n|^2 \leq \frac{1}{4} \sum_{n=1}^{\infty} |E_n|^2,
$$
and  
$$
\sum_{n=1}^{\infty} (n+1)^t |\gamma_n|^2 \leq \frac{1}{4} \sum_{n=1}^{\infty} \frac{(n+1)^t}{n^2} |E_n|^2 \quad \text{for } t \leq 2,
$$
where the coefficients $E_n$ are defined as in \eqref{T-70}. Moreover, the first inequality is sharp.
\end{thm}

\subsection{The class $\mathcal{S}_\alpha$}
For \(-\pi/2 < \alpha < \pi/2\), the class \(\mathcal{S}_{\alpha}\), sometimes known as the Robertson class, is defined as follows
\[
\mathcal{S}_{\alpha} =: \left\{ f \in \mathcal{A} : \operatorname{Re} \left\{ e^{i\alpha} \left( 1 + \frac{z f''(z)}{f'(z)} \right) \right\} > 0\right\}.
\]
%Unlike convex or spirallike functions, members of \(\mathcal{S}_{\alpha}\) are not necessarily univalent. Robertson \cite{Robertson-1969} himself showed that univalence in \(\mathcal{S}_{\alpha}\) is not guaranteed for all \(\alpha\). Later, Pfaltzgraff \cite{Pfaltzgraff1975} proved that functions in \(\mathcal{S}_{\alpha}\) are univalent if $0 < \cos \alpha \leq \frac{1}{2}$.
%The study of \(\mathcal{S}_{\alpha}\) and its generalizations, such as \(\mathcal{S}_{\alpha}(\beta)\) introduced by Pinchuk \cite{Pinchuk1971}, continues to be of interest in geometric function theory, particularly in exploring the boundary between univalent and non-univalent analytic function classes.
For the class $\mathcal{S}_\alpha$, we obtain the following counterpart of Theorem \ref{T-thm-10}.

\begin{thm}\label{T-thm-30}
If $ f \in \mathcal{S}_\alpha ~(|\alpha|<\pi/2)$, then the logarithmic coefficients $\gamma_n$ of $ f $ satisfy the inequalities  
$$
\sum_{n=1}^{\infty} n^2 |\gamma_n|^2 \leq \frac{1}{4} \sum_{n=1}^{\infty} |F_n|^2,
$$
and  
$$
\sum_{n=1}^{\infty} (n+1)^t |\gamma_n|^2 \leq \frac{1}{4} \sum_{n=1}^{\infty} \frac{(n+1)^t}{n^2} |F_n|^2 \quad \text{for } t \leq 2,
$$
where the $F_n$ are Taylor coefficients of the function  
$$
\psi_2(z) = \frac{Az}{(1-z)(1-(1-z)^{A})} =: 1 + \sum_{n=1}^{\infty} F_n z^n,
$$
with $A=e^{-2i\alpha }$. Moreover, the first inequality is sharp.   

\end{thm}

In the next theorem, we obtain sharp estimates of the initial logarithmic coefficients for functions in the class $\mathcal{S}_\alpha$.

\begin{thm}\label{T-thm-40}
If $f \in \mathcal{S}_\alpha$ for $|\alpha| < \pi/2$, then the logarithmic coefficients $\gamma_n$ of $f(z)$ satisfy
     \begin{align*}
  |\gamma_1| \leq \frac{\cos\alpha}{2}, \quad
|\gamma_2| \leq \frac{\cos\alpha}{6}\sqrt{4+5\cos^2\alpha}, \quad
|\gamma_3| \leq \frac{\cos\alpha}{3}\sqrt{1+3\cos^2\alpha}.
 \end{align*}
 All the estimates are sharp.  
\end{thm}

%Recent work has focused on bounds for $\gamma_n$ in families like $\mathcal{S}^*(\varphi)$, $\mathcal{C}(\varphi)$, and subfamilies of $\mathcal{S}$ \cite{cho-Alimohammadi-2021}. Recently, Ponnusamy et al. \cite{Ponnusamy-Sharma- Wirths-2020} proved for $f \in \mathcal{S}^*(A,B)$.\\

%\textbf{Theorem A}\label{T-25} \cite{Ponnusamy-Sharma- Wirths-2020}
%Let $ f \in \mathcal{S}^*(A, B) $ for $-1 \leq B < A \leq 1$. Then the logarithmic coefficients of $ f $ satisfy the inequalities
%
%\begin{equation}
%\sum_{n=1}^{\infty} n^2 |\gamma_n|^2 \leq \frac{(A-B)^2}{4(1-B^2)}, \quad B \neq -1
%\end{equation}
%
%and
%
%$$
%\sum_{n=1}^{\infty} (n+1)^t |\gamma_n|^2 \leq \left(\frac{A-B}{2B}\right)^2 \sum_{n=1}^{\infty} \frac{(n+1)^t}{n^2} |B|^{2n} \quad \text{for } t \leq 2.\\
%$$   
%\vspace{1mm}

\section{Auxiliary Results}\label{T-section-002}
In this section, we list down some results which are known in the literature and are useful in the proof of our results.
Within the class $\mathcal{H}$, consider the subclass $\mathcal{B}$ consisting of functions $\omega$ satisfying $|\omega(z)| < 1$ for all $z \in \mathbb{D}$, and its subclass $\mathcal{B}_0$ where $\omega(0) = 0$. Any $\omega \in \mathcal{B}_0$ admits a Taylor series of the form
\begin{align}\label{T-01}
\omega(z)=\sum_{n=1}^\infty c_n z^n.
\end{align}
With the class $\mathcal{B}_0$, a closely related class is the class $\mathcal{P}$. It consists of analytic functions \(p(z)\) on the unit disk \(\mathbb{D}\) satisfying
\[
p(0) = 1, \qquad \operatorname{Re} p(z) > 0 \quad \text{for all } z \in \mathbb{D}.
\]
Such functions are called \emph{Carathéodory functions}. They appear naturally in the study of univalent functions.\\

In $1981$, Prokhorov et al. \cite{1981-Prokhorov-szynal} obtained a coefficient inequality for functions in $\mathcal{B}_0$.
\begin{lem}\label{T-34}\cite{1981-Prokhorov-szynal}
Let $f\in\mathcal{B}_0$ be of the form \eqref{T-01}. Then for real $\mu$ and $\nu$, we have
\begin{align*}
|c_3+\mu c_1c_2+\nu c_1^3|\le
\begin{cases} 
1 & \text{if } (\mu, \nu) \in D_1 \cup D_2 \cup \{ (2, 1) \} \\
|\nu| & \text{if } (\mu, \nu) \in \bigcup_{k=3}^{7} D_k \\
\frac{2}{3} \left( |\mu| + 1 \right) \left( \frac{|\mu| + 1}{3(|\mu| + 1 +\nu)} \right)^{1/2} & \text{if } (\mu, \nu) \in D_8 \cup D_9 \\
\frac{1}{3} \nu \left( \frac{\mu^2 - 4}{\mu^2 - 4\nu} \right) \left( \frac{\mu^2 - 4}{3(\nu - 1)} \right)^{1/2} & \text{if } (\mu, \nu) \in D_{10} \cup D_{11} - \{ (2, 1) \} \\
\frac{2}{3} \left( |\mu| - 1 \right) \left( \frac{|\mu| - 1}{3(|\mu| - 1 - \nu)} \right)^{1/2} & \text{if } (\mu, \nu) \in D_{12}.
\end{cases}
\end{align*}
where 
\allowdisplaybreaks
\begin{align*}
D_1 &= \left\{ (\mu, \nu): |\mu| \leq \frac{1}{2},\; -1 < \nu \leq 1 \right\}, \\
D_2 &= \left\{ (\mu, \nu): \frac{1}{2} \leq |\mu| \leq 2,\; \frac{4}{27} (|\mu| + 1)^3 - (|\mu| + 1) \leq \nu \leq 1 \right\}, \\
D_3 &= \left\{ (\mu, \nu): |\mu| \leq \frac{1}{2},\; \nu \leq -1 \right\}, \\
D_4 &= \left\{ (\mu, \nu): |\mu| \geq \frac{1}{2},\; \nu \leq -\frac{2}{3} (|\mu| + 1) \right\}, \\
D_5 &= \left\{ (\mu, \nu): |\mu| \leq 2,\; \nu \geq 1 \right\}, \\
D_6 &= \left\{ (\mu, \nu): 2 \leq |\mu| \leq 4,\; \nu \geq \frac{1}{12} (\mu^2 + 8) \right\}, \\
D_7 &= \left\{ (\mu, \nu): |\mu| \geq 4,\; \nu \geq \frac{2}{3} (|\mu| - 1) \right\}, \\
D_8 &= \left\{ (\mu, \nu): \frac{1}{2} \leq |\mu| \leq 2,\; -\frac{2}{3} (|\mu| + 1) \leq \nu \leq \frac{4}{27} (|\mu| + 1)^3 - (|\mu| + 1) \right\}, \\
D_9 &= \left\{ (\mu, \nu): |\mu| \geq 2,\; -\frac{2}{3} (|\mu| + 1) \leq \nu \leq \frac{2|\mu|(|\mu| + 1)}{\mu^2 + 2|\mu| + 4} \right\}, \\
D_{10} &= \left\{ (\mu, \nu): 2 \leq |\mu| \leq 4,\; \frac{2|\mu|(|\mu| + 1)}{\mu^2 + 2|\mu| + 4} \leq \nu \leq \frac{1}{12} (\mu^2 + 2) \right\}, \\
D_{11} &= \left\{ (\mu, \nu): |\mu| \geq 4,\; \frac{2|\mu|(|\mu| + 1)}{\mu^2 + 2|\mu| + 4} \leq \nu \leq \frac{2|\mu|(|\mu| - 1)}{\mu^2 - 2|\mu| + 4} \right\}, \\
D_{12} &= \left\{ (\mu, \nu): |\mu| \geq 4,\; \frac{2|\mu|(|\mu| - 1)}{\mu^2 - 2|\mu| + 4} \leq \nu \leq \frac{2}{3} (|\mu| - 1) \right\}.
\end{align*}

Moreover, all the inequalities are sharp.
\end{lem}

 \begin{lem}\label{T-200}\cite{2016-Iason}
 Let $f\in\mathcal{B}_0$ be of the form \eqref{T-01}. Then for $\lambda \in \mathbb{C} $, we have
\begin{align*}
|c_2+\lambda c_1^2| &\le  \max\{1, |\lambda|\},\\
|c_3+(1+\lambda) c_1c_2+\lambda c_1^3| &\le \max\{1, |\lambda|\}
\end{align*}
and the inequalities are sharp.
 \end{lem}

%\begin{lem}\label{T-23}\cite{Rogosinski-1943}
%  Let $ f(z) = \sum_{n=0}^{\infty} a_n z^n $ and $ g(z) = \sum_{n=0}^{\infty} b_n z^n $ be analytic in $\mathbb{D}$, and $ f \prec g $. Then for each natural number $n\in\mathbb{N}$,
%$$
%%$$  
%\end{lem}

\begin{lem}\label{T-24}\cite{sharma-keong-ali-2024} Let $ f $ be in $ \mathcal{S}^*(\varphi)$, where $\varphi(z)=1+\sum_{n=1}^{\infty}B_nz^n $. Then the logarithmic coefficients $\gamma_n$ of $f$ satisfy the following inequalities
\begin{align*}
\sum_{n=1}^{\infty} n^2 |\gamma_n|^2 &\leq \frac{1}{4} \sum_{n=1}^{\infty} |B_n|^2,\\
\sum_{n=1}^{\infty} (n+1)^t |\gamma_n|^2 &\leq \frac{1}{4} \sum_{n=1}^{\infty} \frac{(n+1)^t}{n^2} |B_n|^2, \quad t \le 2.
\end{align*}
The first inequality is sharp for the function $ f $ given by  $ zf'(z)/f(z)=\varphi(z) $.
\end{lem}

\begin{rem}
The authors in \cite{sharma-keong-ali-2024} proved Lemma \ref{T-24} under the additional hypothesis that $\varphi$ is univalent and starlike with respect to $\varphi(0) = 1$, has a positive real part in $\mathbb{D}$, and $\varphi'(0) > 0$. However, after close inspection, we observed that these conditions are not necessary (see \cite[Theorems 2.2 and 2.3]{sharma-keong-ali-2024}).
\end{rem}

%\begin{lem}\label{T-27}\cite[Th.2.4f.]{miller-mocanu-2000} Let $f, \lambda \in \mathcal{A}$ and $\operatorname{Re(\lambda(z))>0}$ in $\mathbb{D}$, also we have 
%$$
%\operatorname{Re}\left(f(z)+z\lambda(z)f'(z) \right) > 0 \text{ for } z \in \mathbb{D}.
%$$
%Then $\operatorname{Re}f(z)>0.$
%\end{lem}
%\begin{lem}\label{T-33}\cite[Theorem 3.4(a), p.120]{miller-mocanu-2000}
%Let $h$ be analytic in $\mathbb{D}$ and let $\varphi$ be analytic in a domain containing $h(\mathbb{D})$ and suppose
%\begin{enumerate}
 %   \item[(i)] $\operatorname{Re}\left(\varphi (h(z))\right) > 0$, and either
  %  \item[(ii)] $h$ is convex, or
   % \item[(iii)] $H(z) = \varphi (h(z)) \cdot zh'(z)$ is starlike.
%\end{enumerate}
%If $p$ is analytic in $\mathbb{D}$, with $p(0) = h(0)$, $p(\mathbb{D}) \subset \mathbb{D}$ and
%$$
%p(z) + \varphi [p(z)] \cdot zp'(z) \prec h(z),
%$$
%then $p(z) \prec h(z)$.
%\end{lem}

\begin{lem}\label{T-28}\cite[Th.3.2i.]{miller-mocanu-2000} Let $ h $ be convex in $ \mathbb{D} $, with $ h(0) = 1 $. If $ q $ is the analytic solution of
$$
q(z) + \frac{zq'(z)}{q(z)} = h(z), \quad q(0) = 1
$$
and if $ \operatorname{Re} q(z) > 0 $, then $ q $ is univalent. If $ p \in \mathcal{H}$ with $p(0)=1$ satisfies
$$
p(z) + \frac{zp'(z)}{p(z)} \prec h(z),
$$
then $ p \prec q $ and $ q $ is the best dominant. 

%Moreover, if $ F \in \mathcal{A}$, then
%$$
%\frac{zF''(z)}{F'(z)} + 1 \prec h(z) \implies \frac{zF'(z)}{F(z)} \prec %q(z),
%$$
%and this result is sharp.
\end{lem}

\begin{lem}\label{T-29}\cite{miller-mocanu-1985} Let $\beta$ and $\gamma$ be complex numbers with $\beta \neq 0$, and let $h(z) = a + h_1 z + \cdots$ be regular in $\mathbb{D}$. If $\operatorname{Re}[\beta h(z) + \gamma] > 0$ in $\mathbb{D}$ then the solution of
$$
q(z) + \frac{z q'(z)}{\beta q(z)+\gamma} = h(z),\quad q(0) = a
$$
is analytic in $\mathbb{D}$. Further, the solution satisfies $\operatorname{Re}[\beta q(z) + \gamma] > 0$ and is given by
$$
q(z) =
\begin{cases}
H'(z) \left( \beta \int_{0}^{z} H''(t) t^{-1} dt \right)^{-1} - \gamma/\beta & \text{if } a = 0, \\[2mm]
z^{\gamma} [H(z)]^{\beta a} \left( \beta \int_{0}^{z} [H(t)]^{\beta a} t^{\gamma - 1} dt \right)^{-1} - \gamma/\beta & \text{if } a \neq 0,
\end{cases}
$$
where
$$
H(z) =
\begin{cases}
z \exp \left( \frac{\beta}{\gamma} \int_{0}^{z} \frac{h(t)}{t} dt \right) & \text{if } a = 0, \\[2mm]
z \exp \left( \int_{0}^{z} \frac{h(t) - a}{at} dt \right) & \text{if } a \neq 0.
\end{cases}
$$
\end{lem}

The \emph{Gaussian hypergeometric function} \( F(a,b;c;z)= {{}_{2}F_{1}}(a,b;c;z)\) is defined in \(\mathbb{D}\) by the power series expansion
\[
F(a,b;c;z)={{}_{2}F_{1}}(a,b;c;z) = \sum_{n=0}^{\infty} \frac{(a)_n (b)_n}{(c)_n \, n!} \, z^{n},
\]
where \((a)_n\) denotes the \emph{Pochhammer symbol}, i.e., \((a)_0 = 1\) and
\[
(a)_n = a(a+1) \dotsm (a+n-1), \quad n = 1,2,\dots.
\]
Here $a, b,$ and $c$ are complex numbers with \(c \notin -\mathbb{N}_0 := \{0,-1,-2,\dotsc\}\).
By definition, we have \(F(a,b;c;z) = F(b,a;c;z)\).
For further properties of the hypergeometric function, we refer to the handbook-1965-stegun-1965-stegun \cite{handbook-1965-stegun-1965-stegun}. Recently, Sugawa \cite{sugawa-wang-2024} obtained specific conditions on the parameters $a, b, c$ for which  the ratio of two hypergeometric functions is not a convex function. 

\begin{lem}\cite{sugawa-wang-2024}\label{T-43}
    Let $F(a, b; c; z)$ be the hypergeometric function with parameters $a,b,c$ such that $a + b - 1 < c < a + b + 1/2$ and $(c - a)(c - b) > 0$. Then the function
$F(a + 1, b; c; z)/F(a, b; c; z)$ is not convex in $\mathbb{D}$. 
\end{lem}

\section{Proof of the main results}\label{T-section-003}

\begin{proof}[\textbf{Proof of Proposition \ref{T-prop-001}}]
The function $f_c(z)$ defined by \eqref{T-fun-001} can be written
%$\begin{cases}
%\displaystyle \frac{(1 - z)^{1-c} - 1}{c - 1}, &~~~c\neq1,\\
%\displaystyle -\log(1-z), &~~~c=1
%\end{cases
in terms of the Gauss hypergeometric function as
\begin{align*}
f_c(z) = z \, F(c, 1, 2; z).
\end{align*}
%\begin{align*}
%g(z) &= \frac{1}{(c-1)}\sum_{n=1}^{\infty} \frac{(c-1)_n}{n!} z^n \\
%&= z \sum_{n=0}^{\infty} \frac{(c)_n}{(n+1)!} z^n \\
%&= z \, {}_2F_1(c, 1, 2; z)
%\end{align*}
%Therefore, $F_1(z)=z \, {}_2F_1(c, 1, 2; z)$. 
Using this representation, the function $\psi(z)$ defined by \eqref{T-fun-002} can be expressed as
\begin{align*}
\psi(z)
&=\frac{zf_c'(z)}{f_c(z)}
=\frac{z(z \, F(c,1,2,z))'}{z \, F(c,1,2,z)}
= 1 + \frac{z \, F'(c,1,2,z)}{F(c,1,2,z)} \\
&= 1 - c + c\frac{F(c+1,1,2,z)}{F(c,1,2,z)}.
\end{align*}
By Lemma \ref{T-43}, the ratio 
$$\frac{F(c+1,1,2,z)}{F(c,1,2,z)}$$
is not convex for $c\in (1/2,2)$, and consequently, the function  $\psi(z)$ is also not convex for these values of $c$.
\end{proof}

\begin{proof}[\textbf{Proof of Theorem \ref{T-thm-001}}]
Suppose $ f \in \mathcal{F}(c) $ where $c\in(0,2]$. Then 
$$
1 + \frac{zf''(z)}{f'(z)} \prec 1 + \frac{cz}{1-z}=:\varphi(z).
$$
%Since $ c \in (0, 2]  $, the function $ \varphi $ has positive real part in $ \mathbb{D} $ because
%$$
%\inf \left\{ \operatorname{Re} \varphi(z) : z \in \mathbb{D} \right\} = \inf \left\{ 1 + c \operatorname{Re} \frac{z}{1 - z} : z \in \mathbb{D} \right\} = 1 - \frac{c}{2} \geq 0.
%$$
If we take $ g(z) = zf'(z)/f(z) $, which is a non-vanishing analytic function in $\mathbb{D}$ with $ g(0) = 1 $, then from the aforementioned subordination relation, we get
\begin{align}\label{T-31}
g(z) + \frac{zg'(z)}{g(z)} \prec 1 + \frac{cz}{1-z}.
\end{align}
Next, assume that $\psi$ satisfies the differential equation
\begin{align}\label{T-32}
\psi(z) + \frac{z\psi'(z)}{\psi(z)} = \varphi(z). 
\end{align}
Since $\varphi(z)\in \mathcal{P}$, by Lemma \ref{T-29}, the solution $\psi$ that satisfies $\operatorname{Re} \psi(z) > 0$ in $\mathbb{D}$ is given by \eqref{T-fun-002}. From relations \eqref{T-31} and \eqref{T-32}, together with Lemma \ref{T-28}, we find that $\psi$ is univalent for $c \in (0, 2]$ in $\mathbb{D}$ and
\begin{align}\label{T-36}
    g(z) \prec \psi(z) \quad \text{i.e.}, \quad \frac{zf'(z)}{f(z)} \prec \psi(z),  \quad \text{i.e.} \quad \mathcal{F}(c) \subset \mathcal{S}^*(\psi), 
\end{align}
where $\psi$ is the best dominant. Since $\psi(z) = 1 + \sum_{n=1}^{\infty} D_n z^n$, it follows from \eqref{T-36} and Lemma \ref{T-24} that
$$
\sum_{n=1}^{\infty} n^2 |\gamma_n|^2 \leq \frac{1}{4} \sum_{n=1}^{\infty} |D_n|^2
$$
and
$$
\sum_{n=1}^{\infty} (n+1)^t |\gamma_n|^2 \leq \frac{1}{4} \sum_{n=1}^{\infty} \frac{(n+1)^t}{n^2} |D_n|^2, \quad \text{for } t \leq 2.
$$
The first inequality is sharp for the function $f_c(z)$ defined by \eqref{T-fun-001}, because
$$
z \left( \log \left( \frac{f_c(z)}{z} \right) \right)'= \frac{z f_c'(z)}{f_c(z)} - 1 = \psi(z)-1
$$
which implies
$$
\sum_{n=1}^{\infty} 2n \gamma_n(f_c) z^n 
%= \psi(z) - 1 = \frac{c}{2} z + \left( \frac{c^2}{12} + \frac{c}{3} \right) z^2 + \left( \frac{c^2}{8} + \frac{c}{4} \right) z^3 + \cdots 
= \sum_{n=1}^{\infty} D_n z^n.
$$
This completes the proof.
\end{proof}

Before we prove our next result, we note that if $f(z)$ is of the form \eqref{T-05} and the logarithmic coefficients $\gamma_n$ of $f(z)$ are given by \eqref{T-71}, then by differentiating \eqref{T-71} and equating the coefficients of $z, z^2,$ and $z^3$, one can obtain
\begin{align}\label{T-15}
\gamma_1=\frac{1}{2}a_2,\quad \gamma_2=\frac{1}{2}\left(a_3-\frac{1}{2}a_2^2\right),\quad \gamma_3=\frac{1}{2}\left(a_4-a_2a_3+\frac{1}{3}a_2^3\right).
\end{align}

\begin{pf}[\textbf{Proof of Theorem \ref{T-thm-20}}]
Let $ f \in \mathcal{C}(A,B) $ be of the form \eqref{T-05} . Then
$$
1 + \frac{zf''(z)}{f'(z)} \prec \frac{1+Az}{1+Bz}.
$$  
Thus, there exist an $\omega\in\mathcal{B}_0$ of the form \eqref{T-01} such that
$$
1 + \frac{zf''(z)}{f'(z)} = \frac{1+A \omega(z)}{1+B \omega(z)}.
$$
Now equating the coefficients of $z^2,z^3$ and $z^4$, and simplifying, we obtain
    \begin{align*}
    a_2&=\frac{c_1(A-B)}{2}, \quad a_3=\frac{A-B}{6}\left((A-2B)c_1^2+c_2\right),\\
    a_4&= \frac{A-B}{24}\left(2c_3+(3A-7B)c_{1}c_{2}+(A^2-5AB+6B^2)c_1^3\right).
    \end{align*}
Substituting these values into \eqref{T-15}, we obtain
\begin{align}
    \gamma_1&=\frac{c_1(A-B)}{4}, \quad \gamma_2=\frac{A-B}{48}\left((A-5B)c_1^2+4c_2\right),\label{T-91}\\
    \gamma_3&= \frac{A-B}{24}\left(c_3+\frac{1}{2}(A-5B)c_{1}c_{2}+\frac{1}{2}(3B^2-AB)c_1^3\right).\label{T-92}
    \end{align}
Hence from \eqref{T-91}, we find
    $$
     |\gamma_1| \leq \frac{A - B}{4}
    $$
    and the inequality is sharp for the function $g_1\in\mathcal{C}(A,B)$ defined by
\begin{align}\label{T-98}
1 + \frac{zg_1''(z)}{g_1'(z)} = \frac{1+Az}{1+B z}.
\end{align}
Furthermore, from \eqref{T-91} and Lemma \ref{T-200}, we get
\begin{align*}
|\gamma_2|
&=\frac{A-B}{12}\left|c_2+\frac{A-5B}{4}c_1^2\right|\\
&\le \frac{A-B}{12}\max\left\{1,\frac{|A-5B|}{4}\right\}\\
& =\begin{cases}\frac{A-B}{12} & ~~\text{for}~ |A-5B|\le 4\\[2mm]
\frac{A-B}{48}|A-5B| & ~~\text{for}~ |A-5B|>4
 \end{cases}
\end{align*}
The first inequality is sharp for the function $g_2\in\mathcal{C}(A,B)$ defined by
\begin{align}\label{T-93}
1 + \frac{zg_2''(z)}{g_2'(z)} = \frac{1+Az^2}{1+Bz^2}.
\end{align}
For the function $g_2\in\mathcal{C}(A,B)$, it is easy to see that
$$
a_2=0,\quad a_3=\frac{A-B}{6}\quad\text{and}\quad\gamma_2(g_2)=\frac{1}{2}\left(a_3-\frac{1}{2}a_2^2\right)=\frac{A-B}{12}.
$$
The second inequality is sharp for the function $g_1\in\mathcal{C}(A,B)$ defined by \eqref{T-98}. For the function $g_1\in\mathcal{C}(A,B)$, it is easy to see that
$$
a_2=\frac{A-B}{2},\quad a_3=\frac{A^2-3AB+2B^2}{6}
$$
and so 
$$\gamma_2(g_1)=\frac{1}{2}\left(a_3-\frac{1}{2}a_2^2\right)=\frac{A-B}{48}(A-5B).$$

Further, from \eqref{T-92}, we get
\begin{align}\label{T-905}
|\gamma_{3}|
&=\frac{A-B}{24}\left|c_3+\frac{1}{2}(A-5B)c_{1}c_{2}+\frac{1}{2}(3B^2-AB)c_1^3\right|\\
&= \frac{A-B}{24}\left|c_3+\mu c_{1}c_{2}+\nu c_1^3\right|\nonumber,
\end{align}
with \(\mu = \frac{1}{2}(A - 5B)\) and \(\nu = \frac{1}{2}(3B^2 - AB)\). We now use Lemma \ref{T-34} to find the estimate of $\gamma_3$. For \(-1 \leq B < A \leq 1\), one can quickly verify that 
\begin{equation}\label{T-900}
\mu \in (-2, 3]\quad \text{and}\quad \nu \in \left[-\frac{1}{24}, 2\right].
\end{equation}

Suppose that the sets $D_1,D_2,\ldots,D_{12}$ are defined as in Lemma \ref{T-34}.
\begin{itemize}
\item Since \(|\mu| \le 4\), it immediately follows that the sets \(D_7\), \(D_{11}\), and \(D_{12}\) are empty. Similarly, $D_3$ is empty because $\nu \not\le -1$.

\item If $|\mu|\ge \frac{1}{2}$ then
$$-\frac{2}{3}(|\mu| + 1)\le -1 $$
and so the condition $\nu \leq -\frac{2}{3}(|\mu| + 1)$ is incompatible with the admissible range of \(\nu\) given in \eqref{T-900}. Hence, the set $D_4$ is empty.

\item When $2\le |\mu|\le 4$ we must have $\mu\in(2,3]$. In this case,
$$
\frac{2|\mu|(|\mu| + 1)}{\mu^2 + 2|\mu| + 4}> \frac{2(\mu^2 + \mu)}{9 + 6 + 4}> \frac{\mu^2 + 2}{12}.
$$
This shows that \(D_{10}\) is also empty.

\item The conditions for $D_5$ ($|\mu| \le 2$ and $\nu \ge 1$) is equivalent to 
\[
\frac{A-4}{5} \leq B \leq \frac{A - \sqrt{A^2 + 24}}{6}, \quad |A| < 1.
\]
However, for any \(|A| < 1\), we have \(\frac{A-4}{5} > \frac{A - \sqrt{A^2 + 24}}{6}\). Hence, the set \(D_5\) is also empty.

\end{itemize}

On the other hand, it is easy to check that 
\begin{itemize}
\item when \((A,B)=(0,-0.1)\) then the associated $(\mu,\nu)\in D_1$,
\item when \((A,B)=(-0.5,-0.6)\) then the associated $(\mu,\nu)\in D_2$,
\item when \((A,B)=(0,-0.9)\) then the associated $(\mu,\nu)\in D_6$,
\item when \((A,B)=(0,-0.795)\) then the associated $(\mu,\nu)\in D_8$,
\item when \((A,B)=(0,-0.81)\) then the associated $(\mu,\nu)\in D_9$.
\end{itemize}
This shows that the sets $D_1, D_2, D_6, D_8$ and $D_9$ are non empty and together they cover all possible values of $(A,B)$ satisfying \(-1 \leq B < A \leq 1\).
Now we consider the following cases:\\

\textbf{Case-1:}  Let $(\mu,\nu)\in D_1\cup D_2$. Then from \eqref{T-905} and by Lemma \ref{T-34}, we have
$$
|\gamma_{3}|\le\frac{A-B}{24}.
$$
The inequality is sharp for the function $g_3\in\mathcal{C}(A,B)$ defined by
$$
1 + \frac{zg_3''(z)}{g_3'(z)} = \frac{1+Az^3}{1+Bz^3}.\\
$$
For the function $g_3\in\mathcal{C}(A,B)$, it is easy to see that $a_2=0,~~ a_3=0,~~ a_4=\frac{A-B}{12}$ and so
$$
\gamma_3(g_3)=\frac{1}{2}\left(a_4-a_2a_3+\frac{1}{3}a_2^3\right)=\frac{A-B}{24}.
$$

\textbf{Case-2:} Let $(\mu,\nu)\in D_6$. Then from \eqref{T-905} and by Lemma \ref{T-34}, we have
$$
|\gamma_{3}|\le\frac{A-B}{48}|B(A-3B)|.
$$
The inequality is sharp for the function $g_1\in\mathcal{C}(A,B)$ defined by defined by \eqref{T-98}. For the function $g_1\in\mathcal{C}(A,B)$, it is easy to see that
$$
a_2=\frac{A-B}{2},~~ a_3=\frac{A^2-3AB+2B^2}{6}, ~~a_4=\frac{A^3-6A^2B+11AB^2-6B^3}{24}
$$
and so 
$$\gamma_3(g_1)=\frac{1}{2}\left(a_4-a_2a_3+\frac{1}{3}a_2^3\right)=-\frac{A-B}{48}B(A-3B).$$

\textbf{Case-3:} Let $(\mu,\nu)\in D_8\cup D_9$. Then from \eqref{T-905} and by Lemma \ref{T-34}, we have
$$
|\gamma_3|
\le \frac{A-B}{24}\cdot \frac{2\left( |\mu| + 1 \right)}{3}  \left( \frac{|\mu| + 1}{3(|\mu| + 1 +\nu)} \right)^{1/2}
= \frac{A-B}{72\sqrt{3}} \cdot \frac{\left( |A - 5B| + 2 \right)^{3/2}}{\sqrt{|A - 5B| + 2 + 3B^{2} - AB}}.
$$
The inequality is sharp for the function $g_4\in\mathcal{C}(A,B)$ defined by
$$
1 + \frac{zg_4''(z)}{g_4'(z)} = \frac{1+A\omega(z)}{1+B\omega(z)}
$$
where
\begin{align*}
\omega(z)
&=\frac{z(c-\operatorname{sgn} (\mu)~ z)}{(1-\operatorname{sgn} (\mu)~cz)}, \quad c=\left( \frac{|\mu|+1}{3(|\mu|+\nu+1)}\right)^{1/2}\\[3mm]
&= cz+(c^2-1)\operatorname{sgn} (\mu) z^2+c(c^2-1)z^3+\cdots.
\end{align*}
Here we note that if $(\mu,\nu)\in D_8\cup D_9$ then
$$
3(|\mu|+\nu+1)\ge 3(|\mu|+1)-\frac{2}{3}(|\mu|+1)> |\mu|+1>0
$$
and consequently $|c|<1$. For the function $g_4\in\mathcal{C}(A,B)$, from \eqref{T-905}, we have
%\begin{align*}
%a_2 &= \frac{c(A-B)}{2}, ~~a_3 = \frac{A - B}{6} \left( (A - 2B + 1) c^2 - 1 \right), \\[4pt]
%a_4 &= \frac{A - B}{24} \, c \Bigl( (-3A + 7B - 2) + \bigl( A^2 - 5AB + 3A + 6B^2 - 7B + 2 \bigr) c^2 \Bigr),
%\end{align*}
\begin{align*}
|\gamma_3(g_4)|
&= \frac{A-B}{24}\left|c_3+\mu c_{1}c_{2}+\nu c_1^3\right|\\
&=\frac{A-B}{24}\left|c(c^2-1)+\mu c(c^2-1)\operatorname{sgn} (\mu) +\nu c^3\right|\\
&=\frac{c(A-B)}{24}\left|c^2(1+|\mu|+\nu)-(|\mu|+1)\right|\\
&=\frac{c(A-B)}{24}\left|\frac{1}{3}(|\mu|+1)-(|\mu|+1)\right|\\
&=\frac{c(A-B)}{36}(|\mu|+1)\\
&= \frac{A-B}{72\sqrt{3}} \cdot \frac{\left( |A - 5B| + 2 \right)^{3/2}}{\sqrt{|A - 5B| + 2 + 3B^{2} - AB}}.
\end{align*}

\end{pf}

\begin{proof}[\textbf{Proof of Theorem \ref{T-thm-10}}]
If $ f \in \mathcal{C}(A,B) $, then
$$
1 + \frac{zf''(z)}{f'(z)} \prec \frac{1+Az}{1+Bz}.=:\varphi_1(z).
$$ 
Set $ g(z) = zf'(z)/f(z) $, which is a non-vanishing analytic function in $\mathbb{D}$ with $ g(0) = 1 $. Then
\begin{align}\label{T-38}
    g(z) + \frac{zg'(z)}{g(z)} =1 + \frac{zf''(z)}{f'(z)} \prec \varphi_1(z). 
\end{align}
Assume that $\psi_1$ satisfies the differential equation
\begin{align}\label{T-39}
\psi_1(z) + \frac{z\psi_1'(z)}{\psi_1(z)} = \varphi_1(z). 
\end{align}
Since $-1\leq B< A \leq 1$, the function $ \varphi_1(z)$ has a positive real part in $ \mathbb{D} $.
By Lemma \ref{T-29}, the solution $\psi_1$ satisfying $\operatorname{Re} \psi_1 >0$ in $\mathbb{D}$ is given by $\psi_1=\psi_{A,B}$, where $\psi_{A,B}$ is defined in \eqref{T-70}. Furthermore, using Lemma \ref{T-28}, we find that $\psi_1$ is univalent in $\mathbb{D}$. From relations \eqref{T-38} and \eqref{T-39}, and Lemma \ref{T-28}, we have
\begin{align}\label{T-40}
g(z) \prec \psi_1(z) \quad \text{i.e.} \quad \frac{zf'(z)}{f(z)} \prec \psi_1(z) \quad \text{i.e.}~~~ ~\mathcal{C}(A,B) \subset \mathcal{S}^*(\psi_1). 
\end{align}
Since $\psi_1(z)=\psi_{A,B} = 1 + \sum_{n=1}^{\infty} E_n z^n$, it follows from \eqref{T-40} and Lemma \ref{T-24} that
$$
\sum_{n=1}^{\infty} n^2 |\gamma_n|^2 \leq \frac{1}{4} \sum_{n=1}^{\infty} |E_n|^2
$$
and 
$$
\sum_{n=1}^{\infty} (n+1)^t |\gamma_n|^2 \leq \frac{1}{4} \sum_{n=1}^{\infty} \frac{(n+1)^t}{n^2} |E_n|^2 \quad \text{for } t \leq 2.
$$
The first inequality is sharp for the function $K_{A,B}(z)$ given by \eqref{T-fun-10}, because
$$
z \left( \log \left( \frac{K_{A,B}(z)}{z} \right) \right)'=\frac{z K'_{A,B}(z)}{K_{A,B}(z)} - 1 =\psi_{A,B}-1
$$
which leads to
$$
\sum_{n=1}^{\infty} 2n \gamma_n(K_{A,B}) z^n =\sum_{n=1}^{\infty} E_n z^n.
$$
This completes the proof.
\end{proof}

\begin{pf}[\textbf{Proof of Theorem \ref{T-thm-30}}]
If $ f \in \mathcal{S}_\alpha $ then
$$
1 + \frac{zf''(z)}{f'(z)} \prec \frac{1+e^{-2i\alpha }z}{1-z}=:\varphi_2(z) .
$$  
Set $ g(z) = zf'(z)/f(z) $, which is a non-vanishing analytic function in $\mathbb{D}$ with $ g(0) = 1 $. Then, from the aforementioned subordination, we have
\begin{align}\label{T-112}
    g(z) + \frac{zg'(z)}{g(z)}=1 + \frac{zf''(z)}{f'(z)} \prec \varphi_2(z). 
\end{align}
Assume that $\psi_2$ satisfies the differential equation
\begin{align}\label{T-113}
\psi_2(z) + \frac{z\psi_2'(z)}{\psi_2(z)} =\varphi_2(z).
\end{align}
Since $ \alpha \in (-\pi/2,~ \pi/2)  $, the function $ \varphi_2 $ has a positive real part in $ \mathbb{D} $. By Lemma \ref{T-29}, the solution $\psi_2$ that satisfies $\operatorname{Re} \psi_2 >0$ in $\mathbb{D}$ is given by  
\begin{equation*}
\psi_2(z) =  \frac{Az}{(1-z)(1-(1-z)^{A})} , ~~~A=e^{-2i\alpha}.
\end{equation*}
Furthermore, using Lemma \ref{T-28}, we get $\psi_2$ is univalent in $\mathbb{D}$. From relations \eqref{T-112}  and \eqref{T-113} along with Lemma \ref{T-28}, we have
\begin{align}\label{T-114}
g(z) \prec \psi_2(z) \quad \text{i.e.} \quad \frac{zf'(z)}{f(z)} \prec \psi_2(z) \quad \text{i.e.} ~~~\mathcal{S}_\alpha \subset \mathcal{S}^*(\psi_2), 
\end{align}
Since $\psi_2(z) = 1 + \sum_{n=1}^{\infty} F_n z^n$, it follows from \eqref{T-114} and Lemma \ref{T-24} that
$$
\sum_{n=1}^{\infty} n^2 |\gamma_n|^2 \leq \frac{1}{4} \sum_{n=1}^{\infty} |F_n|^2
$$
and
$$
\sum_{n=1}^{\infty} (n+1)^t |\gamma_n|^2 \leq \frac{1}{4} \sum_{n=1}^{\infty} \frac{(n+1)^t}{n^2} |F_n|^2 \quad \text{for } t \leq 2.
$$

%\begin{align}\label{T-114}
%f \in  \mathcal{S}_\alpha \implies f \in \mathcal{S}^*(\psi_2) \quad \text{i.e.} \quad \mathcal{S}_\alpha \subset \mathcal{S}^*(\psi_2). 
%\end{align}
%Then from the definition of $ \mathcal{S}^*(\psi_2) $, it follows that

%\begin{align}\label{T-115}
%z\frac{d}{dz}\left(\log\left(\frac{f(z)}{z}\right)\right) = \frac{zf'(z)}{f(z)} - 1 \prec \psi_2(z) - 1.
%\end{align} 

The inequality is sharp for the function $h_1(z) = \left[(1-z)^{-A}-1 \right]  \in\mathcal{S}_\alpha$, because
$$
z \left ( \log \left ( \frac{h_1(z)}{z} \right) \right)'=\frac{z h_1'(z)}{h_1(z)} - 1 = \psi_2(z) - 1
$$
yielding
$$
\sum_{n=1}^{\infty} 2n \gamma_n(h_1) z^n = \sum_{n=1}^{\infty} F_n z^n.
$$
This completes the proof.
\end{pf}

\begin{pf}[\textbf{Proof of Theorem \ref{T-thm-40}}]
Let $ f \in \mathcal{S}_\alpha $ be of the form \eqref{T-05} . Then
$$
1 + \frac{zf''(z)}{f'(z)} \prec \frac{1+e^{-2i\alpha }z}{1-z} \quad \text{for } z \in \mathbb{D}.
$$  
Thus, there exists $\omega\in\mathcal{B}_0$ of the form \eqref{T-01} such that
$$
1 + \frac{zf''(z)}{f'(z)} =\frac{1+e^{-2i\alpha }\omega(z)}{1-\omega(z)} 
$$  
For simplicity, we write $A=e^{-2i\alpha }$. Equating the coefficients of $z^2,z^3$ and $z^4$ we have
    \begin{align*}
    a_2=&\frac{c_1(1+A)}{2}, \quad a_3=\frac{(1+A)}{6}\left((2+A)c_1^2+c_2\right),\\
    a_4= &\frac{(1+A)}{24}\left(2c_3+(7+3A)c_{1}c_{2}+(6+5A+A^2)c_1^3\right).
    \end{align*}
By \eqref{T-15}, we obtain
\begin{align}
    \gamma_1=&\frac{c_1(1+A)}{4}, \quad \gamma_2=\frac{(1+A)}{24}\left((5+A)c_1^2+4c_2\right),\label{T-101}\\
    \gamma_3= &\frac{(1+A)}{24}\left(2c_3+(5+A)c_{1}c_{2}+(3+A)c_1^3\right).\label{T-201}
    \end{align}
     Hence from \eqref{T-101}, we get
    $$
     |\gamma_1| \leq \frac{|1+A|}{4}= \frac{\cos\alpha}{2}
    $$
    and the inequality is sharp for the function $h_2\in\mathcal{S}_\alpha$ defined by
\begin{align}\label{T-102}
1 + \frac{zh_2''(z)}{h_2'(z)} =  \frac{1+Az}{1-z}.
\end{align}
Again, from \eqref{T-101} and Lemma \ref{T-200}, we get
\begin{align*}
|\gamma_2|
&= \frac{|1+A|}{6} \left|c_2+ \frac{5+A}{4}c_1^2 \right| \le \frac{|1+A|}{6} \max \left\{1, \frac{|5+A|}{4}\right\}\\
& =\frac{\cos\alpha}{3}\sqrt{4+5\cos^2\alpha}.
\end{align*}
The inequality is sharp for the function $h_2\in\mathcal{S}_\alpha$ defined by \eqref{T-102}.\\

Further, from \eqref{T-201} and Lemma \ref{T-200}, we get
\begin{align*}\label{T-104}
|\gamma_{3}|
&=\frac{|1+A|}{12} \left|c_3+\frac{5+A}{2}c_{1}c_{2}+\frac{3+A}{2}c_1^3\right|\\[2mm]
& \le \frac{|1+A|}{12} \max \left\{1, \frac{|3+A|}{2}\right\}\\[2mm]
&= \frac{\cos\alpha}{3}\sqrt{1+3\cos^2\alpha}
\end{align*}
The inequality is sharp for the function $h_2\in\mathcal{S}_\alpha$ defined by \eqref{T-102}.

\end{pf}

%%%%%%%%%%%%%%%%%%%%%%%%%%%%%%%%%%%%%%%%%%%%%%
\vspace{4mm}

\noindent\textbf{Data availability:}
Data sharing is not applicable to this article as no data sets were generated or analyzed during the current study.\vspace{2mm}

\noindent\textbf{Authors Contributions:}
All authors contributed equally to the investigation of the problem and the order of the authors is given alphabetically according to the surname. All authors read and approved the final manuscript.\vspace{2mm}

\noindent\textbf{Acknowledgement:}
The second author thanks the CSIR for the financial support through CSIR-SRF Fellowship ( File no. : 09/0973(13731)/2022-EMR-I ).

\end{document}